\title{\bf Subgraph densities in signed graphons and the local Sidorenko conjecture}
\author{L\'aszl\'o Lov\'asz\footnote{Research supported by OTKA Grant No.~67867 and ERC Grant No.~227701.}\\
Institute of Mathematics, E\"otv\"os Lor\'and University\\
Budapest, Hungary}

\documentclass{article}
\usepackage{amssymb,amsmath,hyperref,graphicx}
\usepackage{theorem}

\nonstopmode

\newtheorem{theorem}{Theorem}[section]
\newtheorem{prop}[theorem]{Proposition}
\newtheorem{lemma}[theorem]{Lemma}
\newtheorem{claim}{Claim}
\newtheorem{corollary}[theorem]{Corollary}

\theorembodyfont{\rmfamily}
\newtheorem{example}{Example}

\newenvironment{proof}{\medskip\noindent{\bf Proof. }}{\hfill$\square$\medskip}
\newenvironment{proof*}[1]{\medskip\noindent{\bf Proof of #1.}}{\hfill$\square$\medskip}

\long\def\killtext#1{}

\begin{document}

\addtolength{\textwidth}{1in} \addtolength{\textheight}{1in}
\addtolength{\baselineskip}{3pt} \setlength{\oddsidemargin}{0.2in}

\def\hom{{\rm hom}}
\def\iso{{\rm iso}}
\def\inj{{\rm inj}}
\def\surj{{\rm sur}}
\def\Inj{{\rm Inj}}
\def\neighb{{\rm neighb}}
\def\inj{{\rm inj}}
\def\ind{{\rm ind}}
\def\sur{{\rm sur}}
\def\PAG{{\rm PAG}}
\def\mul{{\rm mult}}
\def\flat{{\rm flat}}
\def\supp{{\rm supp}}
\def\cov{{\sf cov}}
\def\eps{\varepsilon}
\def\thet{\vartheta}

\def\CUT{\text{\rm CUT}}
\def\IO{{\infty\to1}}
\def\GR{\text{\rm GR}}
\def\CLQ{\text{\rm CLQ}}
\def\LT{{\text{\rm LEFT}}}
\def\RT{{\text{\rm RIGHT}}}
\def\Haus{{\rm Hf}}

\def\Fi{\mathbf{\Phi}}

\def\maxcut{{\sf maxcut}}
\def\MaxCut{{\sf MaxCut}}
\def\id{{\rm id}}
\def\irreg{{\rm irreg}}

\def\tv{{\rm tv}}
\def\tr{{\rm tr}}
\def\cost{\hbox{\rm cost}}
\def\val{\hbox{\rm val}}
\def\rk{{\rm rk}}
\def\diam{{\rm diam}}
\def\Ker{{\rm Ker}}
\def\simi{{\rm sim}}

\def\eul{\text{\sf eul}}
\def\Eu{\text{\sf euln}}
\def\Tu{\text{\sf tut}}
\def\Cr{\text{\sf chr}}
\def\Fl{\text{\sf flo}}
\def\Fs{\text{\sf f}}
\def\Ham{\text{\sf ham}}
\def\Pmg{\text{\sf pmg}}
\def\Ecp{\text{\sf ecp}}
\def\Mch{\text{\sf match}}
\def\Edge{\text{\sf edge}}
\def\Subg{\text{\sf subg}}
\def\Loop{\text{\sf loop}}
\def\Sflow{\text{\sf sflo}}

\def\Pr{{\sf P}}
\def\E{{\sf E}}
\def\Var{{\sf Var}}
\def\Ent{{\sf Ent}}
\def\PD{{\sf Pd}}
\def\SYM{{\sf Sym}}

\def\T{{^\top}}
\def\wt{\widetilde}
\def\wh{\widehat}

\def\unl#1{(#1)^\text{unl}}
\def\sbd#1{#1^\text{\rm sub}}
\def\ns#1{{{\rm NS}_{#1}}}
\def\gprod{\ast}
\def\gpow{\ast}

\def\AA{\mathcal{A}}\def\BB{\mathcal{B}}\def\CC{\mathcal{C}}
\def\DD{\mathcal{D}}\def\EE{\mathcal{E}}\def\FF{\mathcal{F}}
\def\GG{\mathcal{G}}\def\HH{\mathcal{H}}\def\II{\mathcal{I}}
\def\JJ{\mathcal{J}}\def\KK{\mathcal{K}}\def\LL{\mathcal{L}}
\def\MM{\mathcal{M}}\def\NN{\mathcal{N}}\def\OO{\mathcal{O}}
\def\PP{\mathcal{P}}\def\QQ{\mathcal{Q}}\def\RR{\mathcal{R}}
\def\SS{\mathcal{S}}\def\TT{\mathcal{T}}\def\UU{\mathcal{U}}
\def\VV{\mathcal{V}}\def\WW{\mathcal{W}}\def\XX{\mathcal{X}}
\def\YY{\mathcal{Y}}\def\ZZ{\mathcal{Z}}

\def\Ab{\mathbf{A}}\def\Bb{\mathbf{B}}\def\Cb{\mathbf{C}}
\def\Db{\mathbf{D}}\def\Eb{\mathbf{E}}\def\Fb{\mathbf{F}}
\def\Gb{\mathbf{G}}\def\Hb{\mathbf{H}}\def\Ib{\mathbf{I}}
\def\Jb{\mathbf{J}}\def\Kb{\mathbf{K}}\def\Lb{\mathbf{L}}
\def\Mb{\mathbf{M}}\def\Nb{\mathbf{N}}\def\Ob{\mathbf{O}}
\def\Pb{\mathbf{P}}\def\Qb{\mathbf{Q}}\def\Rb{\mathbf{R}}
\def\Sb{\mathbf{S}}\def\Tb{\mathbf{T}}\def\Ub{\mathbf{U}}
\def\Vb{\mathbf{V}}\def\Wb{\mathbf{W}}\def\Xb{\mathbf{X}}
\def\Yb{\mathbf{Y}}\def\Zb{\mathbf{Z}}

\def\ab{\mathbf{a}}\def\bb{\mathbf{b}}\def\cb{\mathbf{c}}
\def\db{\mathbf{d}}\def\eb{\mathbf{e}}\def\fb{\mathbf{f}}
\def\gb{\mathbf{g}}\def\hb{\mathbf{h}}\def\ib{\mathbf{i}}
\def\jb{\mathbf{j}}\def\kb{\mathbf{k}}\def\lb{\mathbf{l}}
\def\mb{\mathbf{m}}\def\nb{\mathbf{n}}\def\ob{\mathbf{o}}
\def\pb{\mathbf{p}}\def\qb{\mathbf{q}}\def\rb{\mathbf{r}}
\def\sb{\mathbf{s}}\def\tb{\mathbf{t}}\def\ub{\mathbf{u}}
\def\vb{\mathbf{v}}\def\wb{\mathbf{w}}\def\xb{\mathbf{x}}
\def\yb{\mathbf{y}}\def\zb{\mathbf{z}}

\def\Abb{\mathbb{A}}\def\Bbb{\mathbb{B}}\def\Cbb{\mathbb{C}}
\def\Dbb{\mathbb{D}}\def\Ebb{\mathbb{E}}\def\Fbb{\mathbb{F}}
\def\Gbb{\mathbb{G}}\def\Hbb{\mathbb{H}}\def\Ibb{\mathbb{I}}
\def\Jbb{\mathbb{J}}\def\Kbb{\mathbb{K}}\def\Lbb{\mathbb{L}}
\def\Mbb{\mathbb{M}}\def\Nbb{\mathbb{N}}\def\Obb{\mathbb{O}}
\def\Pbb{\mathbb{P}}\def\Qbb{\mathbb{Q}}\def\Rbb{\mathbb{R}}
\def\Sbb{\mathbb{S}}\def\Tbb{\mathbb{T}}\def\Ubb{\mathbb{U}}
\def\Vbb{\mathbb{V}}\def\Wbb{\mathbb{W}}\def\Xbb{\mathbb{X}}
\def\Ybb{\mathbb{Y}}\def\Zbb{\mathbb{Z}}

\def\R{{\mathbb R}}
\def\Q{{\mathbb Q}}
\def\Z{{\mathbb Z}}
\def\N{{\mathbb N}}
\def\C{{\mathbb C}}
\def\U{{\mathbb U}}
\def\Ge{{\mathbb G}}
\def\Ha{{\mathbb H}}

\def\one{{\sf\bf 1}}

\maketitle

\tableofcontents

\section{Introduction}

Let $F$ be a bipartite graph with $k$ nodes and $l$ edges and let $G$
be any graph with $n$ nodes and $m=p\binom{n}{2}$ edges. Sidorenko
\cite{Sid1,Sid2} conjectured that the number of copies of $F$ in $G$
is at least $p^l\binom{n}{k}+o(p^ln^k)$ (where we consider $k$ and
$l$ fixed, and $n\to\infty$). In a weaker form, this was also
conjectured by Simonovits \cite{Sim}.

One can get a cleaner formulation by counting homomorphisms instead
of copies of $F$. Let $\hom(F,G)$ denote the number of homomorphisms
from $F$ into $G$. Since we need this notion for the case when $F$
and $G$ are multigraphs, we count here pairs of maps $\phi:~V(F)\to
V(G)$ and $E(F)\to E(G)$ such that incidence is preserved: if $i\in
V(F)$ is incident with $e\in E(F)$, then $\phi(i)$ is incident with
$\psi(e)$. We will also consider the normalized version
$t(F,G)=\hom(F,G)/n^k$. If $F$ and $G$ are simple, then $t(F,G)$ is
the probability that a random map $\phi:~V(F)\to V(G)$ preserves
adjacency. We call this quantity the {\it density of $F$ in $G$}.

In this language, the conjecture says that $t(F,G)\ge
t(K_2,G)^{|E(F)|}$ (this is an exact inequality, no error terms.) We
can formulate this as an extremal result in two ways: First, for
every graph $G$, among all bipartite graphs with $l$ edges, it is the
graph consisting of disjoint edges (the matching) that has the least
density in $G$. Second, for every bipartite graph $F$, among all
graphs on $n$ nodes and edge density $p$, the random graph $\Ge(n,p)$
has the smallest density of $F$ in it (asymptotically, with large
probability).

Sidorenko proved his conjecture in a number special cases: for trees
$F$, and also for bigraphs $F$ where one of the color classes has at
most $4$ nodes. Since then, the only substantial progress was that
Hatami \cite{Hat} proved the conjecture for cubes.

Sidorenko gave an analytic formulation of this conjecture, which is
perhaps even cleaner, and which we will also use. Let $F$ be a
bipartite graph with bipartition $(A,B)$. For each edge $e\in E(F)$,
let $a(e)$ and $b(e)$ be the endpoints of $e$ in $A$ and $B$,
respectively. Assign a real variable $x_i$ to each $i\in A$ and a
real variable $x_j$ to each $j\in B$. Let $W:~[0,1]^2\to\R_+$ be a
bounded measurable function, and define
\[
t(F,W)=\int_{[0,1]^{V(F)}} \prod_{e\in E(F)} W(x_{a(e)},y_{b(e)})
\prod_{i\in A}dx_i\prod_{j\in B}dy_j.
\]

Every graph $G$ can be represented by a function $W_G$: Let
$V(G)=\{1,\dots,n\}$. Split the interval $[0,1]$ into $n$ equal
intervals $J_1,\dots,J_n$, and for $x\in J_i, y\in J_j$ define
$W_G(x,y)=\one_{ij\in E(G)}$. (The function obtained this way is
symmetric.) Then we have
\[
t(F,G)=t(F,W_G).
\]

In this analytic language, the conjecture says that for every
bipartite graph $F$ and function $W\in\WW_+$,
\begin{equation}\label{EQ:CONJ}
t(F,W)\ge t(K_2,W)^{|E(F)|}.
\end{equation}
Since both sides are homogeneous in $W$ of the same degree, we can
scale $W$ and assume that
\[
t(K_2,W)=\int_{[0,1]^2}W(x,y)\,dx\,dy=1.
\]
Then we want to conclude that $t(F,W)\ge 1$. In other words, the
function $W\equiv 1$ minimizes $t(F,W)$ among all functions $W\ge 0$
with $\int W=1$.

The goal of this paper is to prove that this holds locally, i.e., for
functions $W$ sufficiently close to $1$. Most of the time we will
work with the function $U=W-1$, which can negative values. Most of
our work will concern estimates for the values $t(F',U)$ for various
(bipartite) graphs $F'$. This type of question seems to have some
interest on its own, because it can be considered as an extension of
extremal graph theory to signed graphs.

\subsection{Notation}

For each bigraph, we fix a bipartition and specify a first and second
bipartition class. So the complete bigraphs $K_{a,b}$ and $K_{b,a}$
are different. If $F$ is a bigraph, then we denote by $F\T$ the
bigraph obtained by interchanging the classes of $F$.

We have to consider graphs that are partially labeled. More
precisely, a {\it $k$-labeled graph} $F$ has a subset $S\subseteq
V(F)$ of $k$ elements labeled $1,\dots,k$ (it can have any number of
unlabeled nodes. For some basic graphs, it is good to introduce
notation for some of their labeled versions. Let $P_n$ denote the
unlabeled path with $n$ nodes (so, with $n-1$ edges). Let $P'_n$
denote the path $P_n$ with one of its endpoints labeled. Let $P''_n$
denote the $P_n$ with both of its endpoints labeled. Let $C_n$ denote
the unlabeled cycle with $n$ nodes, and let $C_n'$ be this cycle with
one of its nodes labeled. Let $K_{a,b}$ denote the unlabeled complete
bipartite graph; let $K'_{a,b}$ denote the complete bipartite graph
with its $a$-element bipartition class labeled. Note that
$K_{2,2}\cong C_4$, but $K_{2,2}'$ and $C_4'$ are different as
partially labeled graphs.

The most important use of partial labeling is to define a {\it
product}: if $F$ and $G$ are $k$-labeled graphs, then $FG$ denotes
the $k$-labeled graph obtained by taking their disjoint union and
identifying nodes with the same label.

We set $I=[0,1]$. For a $k$-labeled graph $F$, $\unl{F}$ is the graph
obtained by unlabeling. We set $F_1\gprod F_2=\unl{F_1F_2}$, and
$F^{\gpow2}=F\gprod F$.

Let $\WW$ denote the set of bounded measurable functions
$U:~I^2\to\R$; $\WW_+$ is the set of bounded measurable functions
$U:~I^2\to\R_+$, and $\WW_1$ is the set of measurable functions
$U:~I^2\to[-1,1]$. Every function $U\in\WW$ defines a kernel operator
$L_1(f)\to L_1(f)$ by
\[
f \mapsto \int_I U(.,y)f(y)\,dy.
\]

For $U,W\in\WW$, we denote by $U\circ W$ the function
\[
(U\circ W)(x,y)=\int_I U(x,z)W(z,y)\,dz
\]
(this corresponds to the product of $U$ and $W$ as kernel operators).
For every $W\in\WW$, we denote by $W\T$ the function obtained by
interchanging the variables in $W$.

\subsection{Norms}

We consider various norms on the space $\WW$. We need the standard
$L_2$ and $L_\infty$ norms
\[
\|U\|_2=\Bigl(\int_{I^2} U(x,y)^2\,dx\,dy\Bigr)^2, \qquad
\|U\|_\infty=\text{sup ess} |U(x,y)|.
\]
For graph theory, the {\it cut norm} is very useful:
\[
\|U\|_\square = \sup_{S,T\subseteq I} \Bigl|\int_{S\times T}
U(x,y)\,dx\,dy\Bigr|.
\]
This norm is only a factor less than $4$ away from the operator norm
of $U$ as a kernel operator $L_\infty(I)\to L_1(I)$.

The functional $t(F,U)$ can be used define further useful norms. It
is trivial that $t(C_2,U)^{1/2}=\|U\|_2$. The value
$t(C_{2r},U)^{1/(2r)}$ is the $r$-th {\it Schatten norm} of the
kernel operator defined by $U$. It was proved in \cite{BCLSV1} that
for $U\in\WW_1$,
\[
\|U\|_\square^4\le t(C_4,U)\le \|U\|_\square.
\]
The other Schatten norms also define the same topology on $\WW_1$ as
the cut norm (cf. Corollary \ref{COR:4CYCLE}).

It is a natural question for which graphs does $t(F,W)^{1/|E(F)|}$ or
$t(F,|W|)^{1/|E(F)|}$ define a norm on $\WW$. Besides even cycles and
complete bipartite graphs, a remarkable class was found by Hatami: he
proved that $t(F,|W|)^{1/|E(F)|}$ is a norm if $F$ is a cube. He in
fact proved that Sidorenko's conjecture is true whenever $F$ is such
a ``norming'' graph. However, a characterization of such graphs is
open.

\section{Density inequalities for signed graphons}

\subsection{Ordering signed graphons}

For two bipartite multigraphs $F$ and $G$, we say that $F\le G$ if
$t(F,U)\le t(G,U)$ for all $U\in\WW_1$. We say that $G\ge 0$ if
$t(G,U)\ge0$ for all $U\in\WW_1$. Note that if $U$ is nonnegative,
then trivially $G\subseteq F$ implies that $t(F,U)\le t(G,U)$; but
since we allow negative values, such an implication does not hold in
general. For example, $F\ge 0$ cannot hold for any bigraph $F$ with
an odd number of edges, since then $t(F,-U)=-t(F,U)$.

We start with some simple facts about this partial order on graphs.

\begin{prop}\label{PROP:TRIV}
If $F$ and $G$ are nonisomorphic bigraphs without isolated nodes such
that $F\le G$, then $|E(F)|\ge|E(G)|$, $|E(G)|$ is even, and $G\ge
0$. Furthermore, $|t(F,U)|\le t(G,U)$ for all $U\in\WW_1$.
\end{prop}

The proof of this is based on a technical lemma, which is close to
facts that are well known, but not in the exact form needed here.

\begin{lemma}\label{LEM:NON-ISOM}
Let $F$ and $G$ be nonisomorphic graphs without isolated nodes. Then
for every $U\in\WW_1$ and $\eps>0$ there exists a function
$U'\in\WW_1$ such that $\|U-U'\|_\infty<\eps$ and
$t(F,U')\not=t(G,U')$.
\end{lemma}

\begin{proof}
First we show that if $F$ and $G$ are two graphs without isolated
nodes such that $t(F,W)=t(G,W)$ for every $W\in\WW_1$, then $F\cong
G$. Consider the function $U=\one_{x,y\le1/2}$. Then
$t(F,U)=2^{-|V(F)|}$, so $t(F,U)=t(G,U)$ implies that
$|V(F)|=|V(G)|$. Using the function $U\equiv1/2$, we get similarly
that $|E(F)|=|E(G)|$. Using this, we get (by scaling $W$) that
$t(F,W)=t(G,W)$ for every $W\in\WW$.

For every multigraph $H$ we have
\[
t(F,H)=t(F,W_H) = t(G,W_H)=t(G,H),
\]
and hence it follows that
\[
\hom(F,H)=t(F,H)|V(H)|^{|V(F)|}=t(G,H)|V(G)|^{|V(F)|}=\hom(G,H).
\]
From this it follows by standard arguments that $F\cong G$ (e.g., we
can apply Theorem 1(iii) of \cite{Lov71} to the $2$-partite
structures $(V,E,J)$, where $G=(V,E)$ is a multigraph and $J$ is the
incidence relation between nodes and edges).

Since $F$ and $G$ are non-isomorphic, this argument shows that there
exists a function $W\in\WW_1$ such that $t(F,W)\not=t(G,W)$. The
values $t(F,(1-s)U+sW)$ and $t(F,(1-s)U+sW)$ are polynomials in $s$
that differ for $s=0$. Therefore, there is a value $0\le s\le \eps$
for which they differ. Since $(1-s)U+sW\in\WW_1$ and
$\|U-((1-s)U+sW)\|_\infty = s\|U-W\|_\infty\le\eps$, this proves the
lemma.
\end{proof}

\begin{proof*}{Proposition \ref{PROP:TRIV}}
Applying the definition of $F\le G$ with $U=1/2$, we get that
$2^{-|E(F)|}\le2^{-|E(G)|}$, and hence $|E(F)|\ge|E(G)|$. The
relation $F\le G$ implies that $t(F,U)^2=t(F,U\otimes U)\le
t(G,U\otimes U)=t(G,U)^2$ also holds, so $|t(F,U)|\le |t(G,U)|$ for
all $U\in\WW_1$. By Lemma \ref{LEM:NON-ISOM}, $U$ can be perturbed by
arbitrarily little to get a $U'\in\WW_1$ with $t(F,U')\not= t(G,U')$,
then $t(F,U')<t(G,U')$ and $|t(F,U')|\le |t(G,U')|$ imply that
$t(G,U')>0$. Since $U'$ is arbitrarily close to $U$, this implies
that $t(G,U)\ge 0$, and so $G\ge0$. Since this holds for $U$ replaced
by $-U$, it follows that $G$ must have an even number of edges.
\end{proof*}

\subsection{Edge weighting models}

We need the following generalization of Cauchy--Schwarz:

\begin{lemma}\label{LEM:C-S}
Let $f_1,\dots,f_n:~I^k\to\R$ be bounded measurable functions, and
suppose that for each variable there are at most two functions which
depend on that variable. Then
\[
\int_{I^k} f_1\dots f_n \le \|f_1\|_2\dots\|f_n\|_2.
\]
\end{lemma}

This will follow from an inequality concerning a statistical physics
type model. Let $G=(V,E)$ be a multigraph (without loops), and let
for each $i\in V$ let $f_i\in L_2(I^E)$ such that $f_i$ depends only
on the variables $x_j$ where edge $j$ is incident with node $i$. Let
$f=(f_i:~i\in V)$, and define
\[
\tr(G,f)=\int_{I^E} \prod_{i\in V} f_i(x)\,dx
\]
(where the variables corresponding to the edges not incident with $i$
are dummies in $f_i$).

\begin{lemma}\label{LEM:EDGE-WEIGHT}
For every multigraph $g$ and assignment of functions $f$,
\[
\tr(G,f)\le \prod_{i\in V} \|f_i\|_2.
\]
\end{lemma}

\begin{proof}
By induction on the chromatic number of $G$. Let $V_1,\dots,V_r$ be
the color classes of an optimal coloring of $G$. Let
$S_1=V_1\cup\dots\cup V_{\lfloor r/2\rfloor}$ and $S_2=V\setminus
S_1$. Let $E_0$ be the set of edges between $S_1$ and $S_2$, and let
$E_i$ be the set of edges induced by $S_i$. Let $x_i$ be the vector
formed by the variables in $E_i$. Then
\[
\tr(G,f)=\int_{I^{E_0}} \left(\int_{I^{E_1}}\prod_{i\in S_1}
f_i(x)\,dx_1 \right) \left( \int_{I^{E_2}}
 \prod_{i\in S_2} f_i(x)\,dx_2\right)\,dx_0.
\]
The outer integral can be estimated using Cauchy-Schwarz:
\begin{align}\label{TRPR}
\tr(G,f)^2 &\le \int_{I^{E_0}} \left(\int_{I^{E_1}}\prod_{i\in
S_1} f_i(x)\,dx_1 \right)^2\,dx_0\nonumber\\
&~~~\times\int_{I^{E_0}} \left( \int_{I^{E_2}} \prod_{i\in S_2}
f_i(x)\,dx_2\right)^2\,dx_0.
\end{align}
Let $G_1$ be defined as the graph obtained taking a disjoint copy
$(S_1',E_1')$ of the graph $(S_1,E_1)$, and connect each node $i\in
S_1$ to the corresponding node $i'\in S_1'$ by as many edges as those
joining $i$ to $S_2$ is $G$. Note that these newly added edges
correspond to the edges of $E_0$ in a natural way. We assign to each
node the same function as before, and also the same function (with
differently named variables for the edges in $E_1'$) to $i'$. Then
the first factor in \eqref{TRPR} can be written as
\[
\int_{I^{E_0}} \int_{I^{E_1}}\int_{I^{E'_1}}\prod_{i\in S_1\cup S_1'}
f_i(x)\,dx_1 \,dx_0 = \tr(G_1,f).
\]
We define $G_2$ analogously, and get that the second factor in
\eqref{TRPR} is just $\tr(G_2,f)$. So we have
\begin{equation}\label{EQ:TRCS}
\tr(G,f)^2 \le \tr(G_1,f)\tr(G_2,f)
\end{equation}

Next we remark that $G_1$ and $G_2$ have chromatic number at most
$\lceil r/2\rceil$, and so if $r>2$, then we can apply induction and
use that
\[
\tr(G_j,f) \le \prod_{i\in V(G_j)}\|f_i\|_2 = \prod_{i\in
S_j}\|f_i\|_2^2.
\]
If $r=2$, then $G_j$ has edges connecting pairs $i,i'$ only, and so
\[
\tr(G_j,f) = \prod_{i\in S_j}\|f_i\|_2^2.
\]
In both cases, the inequality in the lemma follows by
\eqref{EQ:TRCS}.
\end{proof}

\subsection{Inequalities between densities}

Let $F_1$ and $F_2$ be two $k$-labeled graphs. Then the
Cauchy--Schwarz inequality implies that for all $U\in\WW$,
\begin{equation}\label{EQ:C-H}
t(F_1\gprod F_2,U)^2 \le t(F_1^{\gpow2},U) t(F_2^{\gpow2},U).
\end{equation}
With the notation introduced above, this can be written as
\begin{equation}\label{EQ:C-H-2}
(F_1\gprod F_2)^2 \le F_1^{\gpow2} F_2^{\gpow2}.
\end{equation}
This also implies that for each $k$-labeled graph $F$,
\begin{equation}\label{EQ:POS}
F^{\gpow2}\ge 0.
\end{equation}

Let $\sbd{F}$ denote the subdivision of graph $F$ with one new node
on each edge.

\begin{lemma}\label{LEM:SUBDIV}
If $F\le G$, then $\sbd{F}\le \sbd{G}$.
\end{lemma}

\begin{proof}
For every $U\in\WW$,
\[
t(\sbd{F},U) = t(F,U\circ U\T) \le t(G,U\circ U\T) = t(\sbd{G},U).
\]
\end{proof}

\begin{lemma}\label{LEM:ATTACHED}
Let $F$ be a bigraph, let $S\subseteq V(F)$, and let $H_1,\dots,H_m$
be the connected components of $F\setminus S$. Assume that each node
in $S$ has neighbors in at most two of the $H_i$. Let $F_i$ denote
the graph consisting of $H_i$, its neighbors in $S$, and the edges
between $H_i$ and $S$. Let us label the nodes of $S$ in every $F_i$.
Then
\[
F^2\le \prod_{i=1}^m F_i^{\gpow2}.
\]
\end{lemma}

\begin{proof}
Let $F_0$ denote the subgraph induced by $S$, and consider the nodes
of $F_0$ labeled $1,\dots,k$; we may assume that these nodes are
labeled the same way in every $F_i$. Then using that $|t_{x_1\dots
x_k}(F_0,U)|\le 1$, we get
\begin{align*}
|t(F,U)|&=\Bigl|\int_{I^k} \prod_{i=0}^m t_{x_1\dots
x_k}(F_i,U)\,dx_1\dots dx_k\Bigr|\\
&\le \int_{I^k} \prod_{i=1}^m |t_{x_1\dots x_k}(F_i,U)|\,dx_1\dots
dx_k.
\end{align*}
Lemma \ref{LEM:C-S} implies the assertion.
\end{proof}

We formulate some special cases.

\begin{corollary}\label{COR:INDEP1}
If $F$ contains two nonadjacent nodes of degree at least $2$, then
$F\le C_4$.
\end{corollary}

More generally,

\begin{corollary}\label{COR:INDEP}
Let $v_1,\dots,v_k$ be independent nodes in $F$ with degrees
$d_1,\dots,d_k$ such that no node of $F$ is adjacent to more than $2$
of them. Then
\[
F^2\le\prod_{i=1}^k K_{2,d_i}\le C_4^{k},
\]
\end{corollary}

A {\it hanging path system} in a graph $F$ is a set
$\{P_1,\dots,P_m\}$ of openly disjoint paths such that the internal
nodes of each $P_i$ have degree $2$, and at most two of them start at
any node. The {\it value} of a hanging path system is the total
number of their internal nodes.

\begin{corollary}\label{COR:HANGING}
Let $F$ be a bigraph that contains a hanging path system with lengths
$r_1,\dots,r_m$. Then
\[
F^2 \le\prod_{i=1}^m C_{2r_i}.
\]
\end{corollary}

Combining with Corollary \ref{COR:CYCLE}, we get the following bound,
which we will use:

\begin{corollary}\label{COR:HANG-BOUND}
Let $F$ be a simple graph that contains a hanging path system of
lengths between $2$ and $r$ and value $2r+a-2$, $a\ge0$. Then $F \le
C_{2r}C_4^{a/2}$.
\end{corollary}

\begin{corollary}\label{COR:ERASE}
Let $F$ be a graph and $S\subseteq V(F)$. Let $F_0$ be obtained by
deleting the edges within $S$, and labeling the nodes in $S$. Then
\[
F\le (F_0^{\gpow2})^{1/2}.
\]
\end{corollary}

\subsection{Special graphs and examples}

\begin{lemma}\label{LEM:MON}
Let $U\in\WW_1$. Then the sequence $(t(C_{2k},U):~k=1,2,\dots)$ is
nonnegative, logconvex, and monotone decreasing.
\end{lemma}

With the notation introduced above, we have $C_2\ge C_4\ge C_6\ge
\dots\ge 0$.

\begin{proof}
We have
\[
t(C_{a+b},U) = \int_{I^2} t_{xy}(P''_a,U)t_{xy}(P''_b,U).
\]
Taking $a=b=k$, nonnegativity follows. Applying Cauchy--Schwarz,
\[
t(C_{a+b},U)\le t(C_{2a},U)^{1/2}t(C_{2b},U)^{1/2}.
\]
This implies logconvexity. Since the sequence remains bounded by $1$,
it follows that the sequence is monotone decreasing.
\end{proof}

\begin{lemma}\label{LEM:GCS}
Let $r_1,r_2,\dots,r_k$ be positive integers, and $r=r_1+\dots+r_k$.
Then
\[
C_r^2\le C_{2r_1}\dots C_{2r_k}.
\]
\end{lemma}

\begin{corollary}\label{COR:CYCLE0}
\[
C_{2k+2}\le C_{2k}C_4^{1/2}.
\]
\end{corollary}

\begin{corollary}\label{COR:CYCLE}
If $1\le r_1,\dots,r_n\le r$ and $\sum_i (r_i-1) = k(r-1)$, then
\[
\prod_{i=1}^k C_{2r_i} \le C_{2r}^k.
\]
\end{corollary}

\begin{corollary}\label{COR:4CYCLE}
For all $k\ge 2$,
\[
C_4^{k-1}\le C_{2k}\le C_4^{k/2}.
\]
\end{corollary}

We can get similar bounds for paths, of which we only state two,
which will be needed. Recall that $P_n$ denotes the path with $n$
nodes and $n-1$ edges.

\begin{lemma}\label{LEM:PATHS}
For all $a,b\ge 1$, we have

\smallskip

{\rm (a)} $P_{a+b+1}\le P_{2a+1}^{1/2}P_{2b+1}^{1/2}$;

\smallskip

{\rm (b)} $P_{2a+b+1}\le P_{2a+1}C_{4b}^{1/4}$.
\end{lemma}

\begin{proof}
Since $P_{a+b+1}=P'_{a+1}\gprod P_{b+1}$, the first inequality
follows by \eqref{EQ:C-H-2}. To get the second, we use the first to
get
\[
P_{2a+b+1}\le P_{2a+1}^{1/2}P_{2a+2b+1}^{1/2}.
\]
Cut $P_{2a+2b+1}$ into pieces $P_{a+1}$, $P_{2b+1}$ and $P_{a+1}$,
and apply Lemma \ref{LEM:ATTACHED}; we get
\[
P_{2a+2b+1}\le P_{2a+1}C_{4b}^{1/2},
\]
and hence
\[
P_{2a+b+1}\le
P_{2a+1}^{1/2}\bigl(P_{2a+1}C_{4b}^{1/2}\bigr)^{1/2}=P_{2a+1}C_{4b}^{1/4}.
\]
\end{proof}

\begin{lemma}\label{LEM:MON2}
Let $U\in\WW_1$. Then for every $h\ge 1$, the sequence
$(t(K_{h,2k},U):~k=1,2,\dots)$ is nonnegative, logconcave and
monotone decreasing.
\end{lemma}

\begin{proof}
The proof is similar, based on the equation
\[
t(K_{h,a+b},U) = \int_{I^h} t_{x_1\dots x_h}(K'_{h,a},U)t_{x_1\dots
x_h}(K'_{h,b},U)\,dx_1\dots dx_h.
\]
\end{proof}

For complete bipartite graphs, however, we don't have a bound similar
to Corollary \ref{COR:4CYCLE}, at least as long as we restrict
ourselves to simple graphs (see Example \ref{EXA:KNN}). But we do
have the following inequality.

\begin{lemma}\label{LEM:KNN}
For all $n\ge 3$, we have
\[
K_{n,n}\le K_{2,n} C_2^{1/2}.
\]
\end{lemma}

\begin{proof}
Let $H$ be the 2-labeled graph obtained from $K_{n,n}$ by deleting an
edge and labeling its endpoints. Then $K_{n,n}=\unl{K_2''H}$, and
hence
\[
K_{n,n}^2 \le (K_2'')^{\gpow2} H^{\gpow2}  = C_2 H^{\gpow2}.
\]
Now taking two unlabeled nodes from one color class from one copy of
$H$ and two unlabeled nodes from the other color class from the other
copy, we get a set of $4$ independent nodes of degree $n$ such that
no three have a neighbor in common. Hence by Corollary
\ref{COR:INDEP},
\[
H^{\gpow2}\le K_{2,n}^2,
\]
which proves the Lemma.
\end{proof}

\begin{example}\label{EXA:KNN}
Let $U:~[0,1]^2\to[-1,1]$ be defined by
\[
U(x,y)=
  \begin{cases}
    -1, & \text{if $x,y\ge 1/2$}, \\
     & \text{otherwise}.
  \end{cases}
\]
Then it is easy to calculate that for all $n,m\ge 1$,
\[
t(K_{n,m},U) = \frac14.
\]
\end{example}

\begin{lemma}\label{LEM:C4FIX}
For all $U\in\WW$ and $x\in I$,
\[
0\le t_{x}(C_{2r}',U)\le t(C_{4r-4},U)^{1/2}.
\]
\end{lemma}

\begin{proof}
The first inequality follows from the formula
\[
t_{x}(C_{2r}',U) =  \int_{I}t_{ux}(P_{r+1}'',U)^2\,du.
\]
For the second, write
\[
t_{x}(C_{2r}',U) = \int_{I^2}
U(x,u)t_{uv}(P_{2r-1}'',U)U(v,x)\,du\,dv,
\]
and apply Cauchy--Schwarz:
\begin{align*}
t_{x}(C_4',U)^2 &\le  \int_{I^2} U(x,u)^2U(v,x)^2\,du\,dv
\int_{I^2} t_{uv}(P_{2r-1}'',U)^2\,du\,dv \\
&= t_x(C_2',U)^2 t(C_{4r-4},U) \le t(C_{4r-4},U).
\end{align*}
\end{proof}

\begin{lemma}\label{LEM:P3FIX}
For all $U\in\WW$, $k\ge 4$ and $x,y\in I$,
\[
|t_{xy}(P_k'',U)|\le t(C_{4k-12},U)^{1/4}.
\]
\end{lemma}

\begin{proof}
We can write
\[
t_{xy}(P_k'',U) = \int U(x,u)t_{uy}(P_{k-1}'',U)\,du.
\]
Hence by Cauchy--Schwarz,
\begin{align*}
t_{xy}(P_k'',U)^2 &\le \int U(x,u)^2\,du \int
t_{uy}(P_{k-1}'',U)^2\,du\\
&\le \int t_{uy}(P_{k-1}'',U)^2\,du =t_y(C_{2k-2}',U)\,du.
\end{align*}
Applying Lemma \ref{LEM:C4FIX} the proof follows.
\end{proof}

\subsection{The Main Bounds}

Our main Lemma is the following.

\begin{lemma}\label{LEM:MAIN}
Let $F$ be a bipartite graph with all degrees at least $2$, with
girth $2r$, which is not a single cycle or a complete bipartite
graph. Then $F\le C_{2r}C_4^{1/4}$.
\end{lemma}

Before proving this lemma, we need some preparation. Let $T$ be a
rooted tree. By its min-depth we mean the minimum distance of any
leaf from the root. (As usual, the {\it depth} of $T$ is the maximum
distance of any leaf from the root.) For a rooted tree $T$, we denote
by $T^2$ the graph obtained by taking two copies of $T$ and
identifying leaves corresponding to each other.

\begin{lemma}\label{LEM:TREE1}
Let $T$ be a tree with min-depth $h$ and depth $g$. Then $T^{\gpow2}$
contains a hanging path system with value at least $g+\max(0,h-3)$,
in which the paths are not longer than $\max(g,2)$.
\end{lemma}

\begin{proof}
The proof is by induction on $|V(T)|$. We may assume that the root
has degree $1$, else we can delete all branches but the deepest from
the root. Let $a$ denote the length of the path $P$ in $T$ from the
root $r$ to the first branching point or leaf $v$.

If $P$ ends at a leaf, then the whole tree is a path of length $a=
g=h$. If $a=1$, we get a hanging path in $T^{\gpow2}$ of length $2$,
and so of value $1=1+\max(0,-1)$. If $a\ge 2$, then we can even cut
this into two, and get two hanging paths in $T^{\gpow2}$ of length
$a$, which has value $2a-2\ge a+\max(0,a-3)$.

If $P$ ends at a branching point, then we consider two subtrees
$F_1,F_2$ rooted at $v$ (there may be more), where $F_1$ has depth
$g-a$. Clearly, $F_1$ has mind-depth at least $h-a$ and $F_2$ has
min-depth and depth at least $h-a$. By induction, $F_1^{\gpow2}$ and
$F_2^{\gpow2}$ contain hanging path systems of value $g-a+\max(0,
(h-a)-3)$ and $h-a+\max(0, (h-a)-3)$, respectively. The two systems
together have value at least $g+h-2a$, and they form a valid system
since $v$ (and its mirror image) are contained in at most one path of
each system. If $a=1$, we are done, since clearly $h\ge 2$ and so
$g+h-2\ge g+\max(0,h-3)$.

Assume that $a\ge 2$. Let $F_3$ be obtained from $F_2$ by deleting
its root. By induction, $F_1^{\gpow2}$ contains hanging path systems
of value $g-a+\max(0, h-a-3)$, and $F_3^{\gpow2}$ contains a hanging
path system of value $h-a+\max(0, h-a-4)$. We can add $P$ and its
mirror image, to get a hanging path system of value
\begin{align*}
(g-a) + (h-a-1) + \max(0, h-a-3)+\max(0, h-a-4) +2(a-1)\\
\ge (g-a)+(h-a-1)+2(a-1)=g+h-3 = g+\max(0,h-3),
\end{align*}
since $h\ge a+1\ge 3$. We know that every path constructed lies in
the tree or its mirror image, except for the paths in the case $g=1$,
which are of length $2$.
\end{proof}

\begin{proof*}{Lemma \ref{LEM:MAIN}}
We distinguish several cases.

\medskip

{\bf Case 1.} $r=2$. By hypothesis, $F$ is not a complete bipartite
graph, and hence we can choose nonadjacent nodes $u$ and $v$ from
different bipartition classes. Let $N$ denote the set of neighbors of
$u$, $|N|=d$, and let $F_0$ denote the graph $F-u$ with the neighbors
of $u$ labeled. Then $F\cong F_0 \gprod K_{d,1}'$, and hence by
\eqref{EQ:FF1F2},
\[
F^2 \le F_0^{\gpow2}(K_{d,1}')^{\gpow2} = F_0^{\gpow2} K_{d,2}\le
F_0^{\gpow2}C_4.
\]
Now let $v_1$ and $v_2$ be the two copies of $v$ in $F_0^2$, and $w$,
any third node in the same bipartition class. These three nodes have
no neighbor in common, so by Corollary \ref{COR:INDEP}, we get that
$F_0^{\gpow2}\le C_4^{3/2}$, and so $F\le C_4^{5/4}$.

\medskip

{\bf Case 2.} $F$ is disconnected. If one of the components is not a
single cycle, we can replace $F$ by this component. If $F$ is the
disjoint union of single cycles, then $F\le C_{2r}^2\le C_{2r}C_4$.

So we may assume that $F$ is connected. Then it must have at least
one node of degree larger than $2$.

\medskip

{\bf Case 3.} $F$ has at most one node of degree larger than $2$ in
each color class. Let $u_1$ and $u_2$ be two nodes, one in each color
class, such that all the other nodes have degree $2$. Then $F$ must
consist of one or more odd paths connecting $u_1$ and $u_2$, and even
cycles attached at $u_1$ and/or $u_2$.

If there are even cycles attached at $u_1$ and also at $u_2$, then
$F$ has a hanging path system consisting of $4$ paths of length $r$,
and so $F\le C_{2r}^2\le C_{2r}C_4^{1/2}$ by Lemma
\ref{LEM:ATTACHED}. If (say) $u_1$ but not $u_2$ has a cycle
attached, then there are at least two paths connecting $u_1$ and
$u_2$, and one of them has length at least $r$. This implies that
$F\le C_{2r}^{3/2}\le C_{2r}C_4^{1/2}$.

So we may assume that $F$ consists of openly disjoint paths
connecting $u_1$ and $u_2$. Since $F$ is not a single cycle, there
are at least three paths. Let $a_1\le a_2\le a_3$ be their lengths.
Clearly $a_1+a_2\ge 2r$. If $a_2\ge r+1$, then we have two hanging
paths of length $r+1$, which implies that $F\le C_{2r+2}\le
C_{2r}C_4^{1/2}$. So we may assume that $a_1=a_2=a_3=r$. If $r\ge 4$,
then we can select two of the paths and path of length $2$ disjoint
from them, which gives $F\le C_{2r}C_4^{1/2}$.

So we get to the special case when $F$ consists of $3$ or more paths
of length $3$ connecting $u_1$ and $u_2$. In this case, we use Lemma
\ref{LEM:P3FIX}:
\begin{align*}
t(F,U)&=\int_{I^2} t_{xy}(P_4'',U)^3\,dx\,dy \le t(C_4,U)^{1/4}
\int_{I^2} t_{xy}(P_4'',U)^2\,dx\,dy\\
&= t(C_6,U) t(C_4,U)^{1/4}.
\end{align*}

\medskip

{\bf Case 4.} Suppose that there are two nodes $u_1,u_2$ in the same
bipartition class of $F$ of degree at least $3$.

Let $S_1$ be the set of nodes in $F$ with $d(x,u_1)\le
\min(r-2,d(x,u_2)-2)$, and let $S'_1$ be the set of neighbors of
$S_1$ in $F$. We define $S_2$ and $S_2'$ analogously. Let $F_i$ be
the subgraph induced by $S_i\cup S_i'$.

\begin{claim}\label{CLAIM:TREES}
$F_i$ is a tree with endnode set $S_i'$. Every $x\in S_i'$ satisfies
$d(x,u_1)=\min(r-1,d(x,u_2))$.
\end{claim}

From the fact that $F$ has girth $2r$ it follows that $F_i$ is a
tree. The nodes in $S_i$ are not endnodes of $F_i$, since their
degree in $F$ is at least $2$ and all their neighbors are nodes of
$F_i$. It is also trivial that the nodes in $S_i'$ are endnodes. Let
$x\in S_i'$, then $x\notin S_i$ and hence $d(x,u_1)\ge
\min(r-1,d(x,u_2)-1)$. But $d(x,u_1)$ and $d(x,u_2)$ have the same
parity, and hence it follows that $d(x,u_1)\ge \min(r-1,d(x,u_2))$.
On the other hand, $x$ has a neighbor $y\in S_i$, and hence
$d(x,u_1)\le d(y,u_1)+1\le r-1$, and $d(x,u_1)\le d(y,u_1)+1\le
d(y,u_2)-1\le d(x,u_2)$. This implies that $d(x,u_1)\le
\min(r-1,d(x,u_2))$, which proves the Claim.

\begin{claim}\label{CLAIM:ATTACHED}
There is no edge between $S_1$ and $S_2$.
\end{claim}

Indeed, suppose that $x_1x_2$ is such an edge, $x_i\in S_i$. Then
$d(x_1,u_1)<d(x_2,u_2)$, which by parity means that $d(x_1,u_1)\le
d(x_2,u_2)-2$. But then $d(x_2,u_1)\le d(x_1,u_1)+1 \le
d(x_2,u_2)-1\le d(x_2,u_1)$, showing that $x_2\notin S_2$.

Consider the nodes of $F_i$ in $S'_i$ as labeled. Lemma
\ref{LEM:ATTACHED} implies that
\begin{equation}\label{EQ:FF1F2}
F^2\le F_1^2 F_2^2.
\end{equation}
Hence to complete the proof, it suffices to show that
\begin{equation}\label{EQ:FFC}
F_i^2\le C_{2r}C_4^{1/4}.
\end{equation}
This will follow by Corollary \ref{COR:HANG-BOUND}, if we construct
in $F_i$ a hanging path system of paths of length at most $r$ with
value $2r-1$.

\begin{claim}\label{CLAIM:DEPTH}
Let $y\not=x$ be two leaves of $F_1$. Then $d(r,x)+d(r,y)+d(x,y)\ge
2r$.
\end{claim}

If $d(r,x)=r-1$ or $d(r,y)=r-1$ then this is trivial, so suppose that
$d(r,x),d(r,y)\le r-2$. Then by Claim \ref{CLAIM:TREES}, we must have
$d(x,u_2)=d(x,u_1)$ and $d(y,u_2)=d(y,u_1)$. Going from $x$ to $u_2$
to $y$ and back to $x$ in $F$, we get a closed walk of length
$d(r,x)+d(r,y)+d(x,y)$, which contains a cycle of length no more than
that, which implies the inequality in the Claim.

\begin{claim}\label{CLAIM:MAIN}
Each of $F_1$ and $F_2$ contains a hanging path system of paths with
length at most $r$, with value $2r-2$. At least one of them contains
such a system with value $2r-1$.
\end{claim}

To prove this, we need to distinguish two cases.

\medskip

{\bf Case 4a.} All branches of $F_1$ are single paths. Let $a_1\le
\dots\le a_d$ be their lengths. Claim \ref{CLAIM:DEPTH} implies that
$a_1+a_2\ge r$, so $a_2\ge r/2$. The graph $F_1^2$ consists of paths
$Q_1,\dots,Q_d$ of length $2a_1,\dots, 2a_d$ connecting $u_1$ and its
mirror image $u_1'$. Select subpaths of length $r$ from $Q_2$ and
$Q_3$, this gives a hanging path system of value $2r-2$. If $a_1\ge
2$, then we can add to this a path of length $2$ from $Q_1$ not
containing its endpoints, and we get a path system of value $2r-1$.
So we may assume that $a_1=1$. Then $a_2\ge r-1>r/2$, and so
$2a_2,2a_r>r$. Thus we can select the paths of length $r$ from $Q_2$
and $Q_3$ so that one of them misses $u_1$ and the other one misses
$u_1'$. The we can add $Q_1$ to the system, and conclude as before.

\medskip

{\bf Case 4b.} At least one of the branches of $F_1$, say $A$, is not
a single path. Let $a$ be the length of the path $Q$ from the root
$u_1$ to the first branch point $v$. Let $T_1,T_2$ be two subtrees of
$A$ rooted at $v$, of depth $d_1$ and $d_2$. Let $B$ and $C$ be two
further branches, of depth $b$ and $c$, respectively, where $b\ge c$.
By Claim \ref{CLAIM:DEPTH}, we have
\[
d_1+d_2+a\ge r\qquad\text{and}\qquad b+c\ge r.
\]

We start with a simple computation showing that we can get a hanging
path system in $F_1^2$ of value $2r-2$. If $a=1$, then we choose a
hanging path system from $T_1^2$ of value $d_1$, from $T_2^2$ of
value $d_2$, from $B^2$ of value $b$ and from $C^2$ of value $c$.
This is a total of
\[
d_1+d_2 +b +c \ge 2r -1.
\]
If $a\ge 2$, then we choose a hanging path system from $T_1^2$ of
value $d_1$, from $(T_2-v)^2$ of value $d_2-1$, from $B^2$ of value
$b$ and from $(C-u_1)^2$ of value $c-1$. Leaving out $v$ from $T_2$
and $u_1$ from $C$ allows us to add $Q$ and its mirror image of value
$2(a-1)$. This is a total of
\begin{equation}\label{EQ:2R-2}
d_1+d_2-1 +2(a-1) +b +c-1 \ge 2r +a-4 \ge 2r-2.
\end{equation}
If equality holds in all estimates, then $d_1+d_2+a=r$, $b+c=r$, and
$a=2$. It also follows that $b\le 3$, or else we get a larger system
in $B$. Note that the depth of $A$ is at least $a+1=3$, and $c\le
r/2\le b\le 3$.

If $B$ is a single path, then we can select a hanging path of length
$r$ from $B^2$, of value $r-1>b-1$, and we have gained $1$ relative
to the previous construction. So we may assume that $B$ is not a
single path. Then applying the same argument as above with $A$ and
$B$ interchanged, we get that $b=3$, and the depth of $A$ is also
$3$. Hence $d_1=d_2=1$ and $r=d_1+d_2+a=4$. It follows that
$c=r-b=1$, so $C$ consists of a single edge.

If $u_1$ has degree larger than $3$, then applying the argument to
$A,B$ and a fourth branch $D$, we get that $D$ must have depth $1$,
but this contradicts Claim \ref{CLAIM:DEPTH}. Hence the degree of
$u_1$ is $3$.

If $A$ has at least $3$ leaves, then these must be connected to $u_2$
by disjoint paths of length $3$. Since $u_2$ must be connected to the
endpoint of $C$ as well by Claim \ref{CLAIM:TREES}, we get that $u_2$
has degree at least $4$, and so $F_2\ge C_{2r}C_4^{1/2}$.

So $A$ and similarly $B$ have two leaves, and $F_1$ is a 10-node tree
consisting of a path with $5$ nodes and $2$ endnodes hanging from its
endnodes and $1$ from its middle node. $F_2$ must be the same, or
else we are done. There is only one way to glue two copies of this
tree together at their endnodes to get a graph of girth $8$, and this
yields the subdivision of $K_{3,3}$ (by one  node on each edge). To
settle this single graph, we use that
\[
K_{3,3}\le C_2^{1/2} K_{3,2}\le C_2^{1/2}C_4
\]
by Lemmas \ref{LEM:KNN} and \ref{LEM:MON2}, and so by Lemma
\ref{LEM:SUBDIV}, we have
\[
F=\sbd{K_{3,3}}\le(\sbd{C}_2)^{1/2}\sbd{C}_4 = C_4^{1/2}C_8.
\]

Thus we know that $F_1^2 F_2^2\ge C_{2r}$, and for at least one of
them $F_i^2\ge C_{2r}C_4^{1/2}$, which implies that $F^2\ge
F_1^2F_2^2\ge C_{2r}C_4^{1/4}$.
\end{proof*}

Lemma \ref{LEM:ATTACHED} implies that if $F$ is a graph with two
nonadjacent nodes $u,v$ of degree $1$, then $F \le P_3$. We need a
stronger bound:

\begin{lemma}\label{LEM:2END}
Let $F$ be a graph with two nonadjacent nodes $u,v$ of degree $1$,
which is not a star and has at least $3$ edges. Then $F \le
P_3C_4^{1/4}$.
\end{lemma}

\begin{proof}
\noindent{\bf Case 1.} First, suppose that $F$ has two nonadjacent
nodes $u,v$ of degree $1$ whose neighbors $u'$ and $v'$ are
different. If there is a node $w\not=u,v,u',v'$ of degree $d\ge 2$,
then we can apply Lemma \ref{LEM:ATTACHED} to the stars of $u$, $v$
and $w$, to get
\[
F\le P_3 K_{2,d}^{1/2}\le P_3 C_4^{1/2}.
\]
If there is a node $w\not=u,v,u',v'$ of degree $1$, then a similar
application of Lemma \ref{LEM:ATTACHED} gives that
\[
F\le P_3^{3/2}\le P_3 C_4^{1/4}.
\]
Finally, if $V(F)=\{u,v,u',v'\}$, then $F=P_4$, and the bound follows
from Lemma \ref{LEM:PATHS}(b).

\noindent{\bf Case 2.} Suppose that all nodes of $F$ of degree $1$
have a common neighbor $w$. Let $F_0$ denote the subgraph obtained by
deleting the nodes of degree $1$. Since $F$ is not a star, $F_0$ must
have at least two edges. If $F_0$ has a node not adjacent to $w$,
then we conclude similarly as above. So suppose that $w$ is adjacent
to all the other nodes of $F_0$. Let $F_1$ denote the subgraph formed
by the edges incident with $w$ (a star), with the nodes in $F_0-w$
labeled.
\end{proof}

\begin{lemma}\label{LEM:ONE-END}
Let $F$ be a bigraph with exactly one node of degree $1$ and with
girth $2r$. Then
\[
F\le \frac12 (C_{2r}+P_3)C_4^{1/8} .
\]
\end{lemma}

\begin{proof}
Let $v$ be the unique node of degree $1$. We can write $F\cong
F_0\gprod P_2'$, where $F_0$ is a $1$-labeled graph in which all
nodes except possibly the labeled node $v$ have degrees at least $2$.
By \eqref{EQ:FF1F2}, we get that $F^2\le
(P'_2)^{\gpow2}F_0^{\gpow2}\cong P_3F_0^{\gpow2}$. Here
$F_0^{\gpow2}$ is a graph with girth $2r$ and all degrees at least
$2$. Hence Lemma \ref{LEM:MAIN}, we get $F^2\le P_3 C_{2r}
C_4^{1/4}$. Thus
\begin{align*}
|t(F,U)|&\le\sqrt{t(P_3,U)t(C_{2r},U)}t(C_4,U)^{1/8}\\
&\le \frac12 (t(C_{2r},U)+t(P_3,U))t(C_4,U)^{1/8}.
\end{align*}
\end{proof}

\section{Local Sidorenko Conjecture}

The Sidorenko Conjecture asserts that $t(F,W)$ is minimized by the
function $W\equiv 1$ among all functions $W\ge 0$ with $\int W=1$.
The following theorem asserts that this is true at least locally.

\begin{theorem}\label{THM:CLOSE}
Let $F=(V,E)$ be a simple bigraph. Let $W\in\WW$ with $\int W=1$,
$0\le W\le 2$ and $\|W-1\|_\square \le 2^{-8m}$. Then $t(F,W)\ge 1$.
\end{theorem}

\begin{proof}
We may assume that $F$ is connected, since otherwise, the argument
can be applied to each component. Let $U=W-1$, then we have the
expansion
\begin{equation}\label{EQ:EXPAND}
t(F,W)=\sum_{F'} t(F',U),
\end{equation}
where $F'$ ranges over all spanning subgraphs of $F$. Since isolated
nodes can be ignored, we may instead sum over all subgraphs with no
isolated nodes (including the term $F'=K_0$, the empty graph). One
term is $t(K_0,U)=1$, and every term containing a component
isomorphic to $K_2$ is $0$ since $t(K_2,U)=\int U=0$.

Based on \eqref{EQ:POS}, we can identify two special kinds of
nonnegative terms in \eqref{EQ:EXPAND}, corresponding to copies of
$P_3$ and to cycles in $F$. We show that the remaining terms do not
cancel these, by grouping them appropriately.

(a) For each node $i\in V$, let $\sum_{\nabla(i)}$ denote summation
over all subgraphs $F'$ with at least two edges that consist of edges
incident with $i$. Let $d_i$ denote the degree of $i$ in $F$, assume
that $d_i\ge 2$, and set $t(x)=t_x(K_2',U)$. Then using that $t(x)\ge
-1$ and Bernoulli's Inequality,
\begin{align*}
\sum_{\nabla(i)} t(F',U) &= \int_I \sum_{k=2}^{d_i} \binom{d_i}{k}
t(x)^k\,dx = \int_I (1+t(x))^{d_i}-1-d_it(x)\,dx\\
&\ge \int_I (1+t(x))(1+(d_i-1)t(x))-1-d_it(x)\,dx\\
&= \int_I (d_i-1)t(x)^2\,dx = (d_i-1)t(P_3,U).
\end{align*}
Hence the terms in \eqref{EQ:EXPAND} that correspond to stars sum to
at least
\[
\sum_{\rm stars} t(F',U) \ge (d_i-1)t(P_3,U) = (2m-n)t(P_3,U).
\]

(b) Next, consider those terms $F'$ with at least two endnodes that
are not stars. For such a term we have
\[
|t(F',U)| \le  t(P_3,U) t(C_4,U)^{1/4}\le 2^{-2m} t(P_3,U)
\]
(if there are two nonadjacent endpoints, then the left hand side is
$0$; else, this follows from Lemma \ref{LEM:2END}). The sum of these
terms is, in absolute value, at most
\[
2^m 2^{-2m} t(P_3,U)<t(P_3,U).
\]

(c) The next special sum we consider consists of complete bipartite
graphs that are not stars. Fixing a subset $A$ with $|A|\ge 2$ in the
first bipartition class of $F$ with $h\ge 2$ common neighbors, and
fixing the variables in $A$, the sum over such complete bigraphs with
$A$ as one of the bipartition classes is
\[
\sum_{j=2}^h \binom{h}{j}\left(\int_I \prod_{i\in A}
U(x_i,y)\,dy\right)^j \ge (h-1)\left(\int_I \prod_{i\in A}
U(x_i,y)\,dy\right)^2
\]
by the same computation as above. This gives that this sum is
nonnegative.

(d) If $F'$ has all degrees at least $2$ and girth $2r$, and it is
not a single cycle or complete bipartite, then $F'\le
C_{2r}C_4^{1/4}$ by Lemma \ref{LEM:MAIN}, and so
\[
|t(F',U)|\le t(C_{2r},U)t(C_4,U)^{1/4}\le 2^{-2m} t(C_{2r},U).
\]
So if we fix $r$ and sum over all such subgraphs, we get, in absolute
value, at most
\[
2^m 2^{-2m} t(C_{2r},U)< \frac12 t(C_{2r},U).
\]

(e) Finally, if $F'$ has exactly one node of degree $1$ and girth
$2r$, then by Lemma \ref{LEM:ONE-END}
\begin{align*}
|t(F',U)|&\le \frac12(t(P_3,U)+t(C_{2r},U))t(C_4,U)^{1/8}\\
&\le 2^{-m-1}(t(P_3,U)+ t(C_{2r},U)).
\end{align*}
If we sum over all such subgraphs $F'$, then we get less than
$t(P_3,U)+\frac12\sum_{r\ge 2}t(C_{2r},U)$.

The sum in (a) is sufficient to compensate for the sum in (b) and the
first term in (e), while the sum over cycles compensates for the sum
in (d) and the second sum in (e). This proves that the total sum in
\eqref{EQ:EXPAND} is nonnegative.
\end{proof}

\subsection{Variations}

We could do the computations above more carefully, and use the
Neumann-Schatten norm $t(C_4,U)^{1/4}$ instead of the cut norm in the
statement. The best one can achieve this way is to replace the bound
of $2^{-8m}$ by about $2^{-m}$:

\begin{theorem}\label{THM:CLOSE2}
Let $F=(V,E)$ be a simple bigraph. Let $W\in\WW$ with $\int W=1$ and
$0\le W\le 2$ and $t(C_4,W) \le 2^{-4m}$. Then $t(F,W)\ge 1$.
\end{theorem}

One can combine the conditions and assume a bound on
$\|W-1\|_\infty$. It follows from the Theorem that $\|W-1\|_\infty\le
2^{-8m}$ suffices. Going through the same arguments (in fact, in a
somewhat simpler form) we get:

\begin{theorem}\label{THM:LOC-SID-INFTY}
Let $F=(V,E)$ be a simple bigraph. Let $W\in\WW$ with $\int W=1$ and
$\|W-1\|_\infty \le 1/(4m)$. Then $t(F,W)\ge 1$.
\end{theorem}

The condition that $\|W-1\|_\infty \le 1/(4m)$ implies trivially that
$0\le W\le 2$. It would be interesting to get rid of the condition
that $W\le 2$ under an appropriate bound on $\|W-1\|_\square$. We can
only offer the following result.

\begin{theorem}\label{THM:LOC-SID-REG}
Let $F=(V,E)$ be a simple bigraph with $m$ edges, let $0<\eps<
2^{-1-8m}$, and let $W\in\WW$ such that $\int W=1$, $\int_{S\times T}
W\le 2\lambda(S)\lambda(T)$ whenever $\lambda(S),\lambda(T)\ge
2^{-4/\eps^2}$, and $\|W-1\|_\square \le 2^{-1-8m}$. Then $t(F,W)\ge
1-\eps$.
\end{theorem}

\begin{proof}
For every function $W\in\WW$ and partition $\PP=\{V_1,\dots,V_k\}$ of
$I$ into a finite number of measurable sets with positive measure,
let $W_\PP$ denote the function obtained by averaging $W$ over the
partition classes; more precisely, we define
\[
W_\PP(x,y)=\frac{1}{\lambda(V_i)\lambda(V_j)}\int_{V_i\times
V_j}W(u,v)\,du\,dv
\]
for $x\in V_i$ and $y\in V_j$.

The Weak Regularity Lemma of Frieze and Kannan in the form used in
\cite{BCLSV1} implies that there is a partition $\PP$ into $K\le
2^{4l^2/\eps^2}$ equal measurable sets such that the function $W_\PP$
satisfies
\[
\|W_\PP-W\|_\square\le \frac{\eps}{m},
\]
and hence by the Counting Lemma 4.1 in \cite{LSz1},
\[
|t(F,W_\PP)-t(F,W)|\le \eps.
\]
Clearly $\int W_\PP=1$, $W_\PP\ge 0$, and for all $x\in V_i$ and
$y\in V_j$,
\[
W_\PP(x,y) = \frac{1}{\lambda(V_i)\lambda(V_j)}\int_{V_i\times
V_j}W(u,v)\,du\,dv \le 2.
\]
Furthermore,
\[
\|W_\PP-1\|_\square \le \|W_\PP-W\|_\square +\|W-1\|_\square \le
2^{-8m},
\]
Thus Theorem \ref{THM:CLOSE} implies that $t(F,W_\PP)\ge 1$, and
hence $t(F,W)\ge t(F,W_\PP)-\eps \ge 1-\eps$.
\end{proof}

\subsection{Graphic form}

We end with a graph-theoretic consequence of Theorem \ref{THM:CLOSE}.

\begin{corollary}\label{COR:GR-CLOSE}
Let $F$ be a bipartite graph with $n$ nodes and $m$ edges, and let
$G$ be a graph with $N$ nodes and $M=p\binom{N}{2}$ edges. Let
$\eps>0$. Assume that
\[
\bigl|e_G(S,T)-p|S|\cdot|T|\bigr|\le (2^{-8m}p-\eps)N^2
\]
for all $S,T\subseteq V(G)$, and
\[
e_G(S,T)\le 2p |S|\cdot|T|
\]
for all $S,T\subseteq V(G)$ with $|S|,|T|\ge 2^{-4m^2/\eps^2}N$. Then
\[
t(F,G)\ge p^l-\eps.
\]
\end{corollary}

\begin{proof}
This follows by applying Theorem \ref{THM:LOC-SID-REG} to the
function $W_G/p$.
\end{proof}

\end{document}